\documentclass[10pt]{scrartcl}
\usepackage{amsmath,amsthm,amssymb,mathtools}
\usepackage[utf8]{inputenc}
\usepackage[T1]{fontenc}
\usepackage{hyperref}
\hypersetup{
  colorlinks=true,
  linkcolor=black,
  citecolor=black,
  urlcolor=blue,
  pdftitle={Long-term behaviour in a parabolic—elliptic chemotaxis—consumption model},
  pdfauthor={Fuest, Lankeit, Mizukami},
  pdfkeywords={chemotaxis; chemotaxis-consumption; global existence; boundedness; 
large-time behaviour; realistic oxygen boundary conditions},
  bookmarksopen=true,
}

\RequirePackage{geometry}
\newtheorem{thm}{Theorem}[section]
\newtheorem{lemma}[thm]{Lemma}

\newtheorem{remark}[thm]{Remark}

\newcommand{\N}{\mathbb{N}}

\newcommand{\ur}[1]{\mathrm{#1}}
\newcommand{\ure}{\ur e}

\newcommand{\Om}{\Omega}
\newcommand{\Ombar}{\overline{\Omega}}
\newcommand{\nn}{\nonumber}

\usepackage{newunicodechar}
\newunicodechar{α}{\alpha}
\newunicodechar{β}{\beta}
\newunicodechar{δ}{\delta}
\newunicodechar{Ω}{\Omega}
\newunicodechar{ℝ}{\mathbb{R}}
\newunicodechar{Φ}{\Phi}
\newunicodechar{ψ}{\psi}
\newunicodechar{Δ}{\Delta}
\newunicodechar{∇}{\nabla}
\newunicodechar{∂}{\partial}
\newunicodechar{ν}{\nu}
\newunicodechar{γ}{\gamma}
\newunicodechar{Λ}{\Lambda}
\newunicodechar{∞}{\infty}
\newunicodechar{ε}{\varepsilon}
\newunicodechar{φ}{\varphi}
\newunicodechar{σ}{\sigma}
\newunicodechar{χ}{\chi}
\newunicodechar{κ}{\kappa}
\newunicodechar{μ}{\mu}
\newunicodechar{τ}{\tau}
\newunicodechar{ℕ}{\mathbb{N}}
\newunicodechar{ζ}{\zeta}
\newunicodechar{λ}{\lambda}

\newcommand{\Lom}[1]{L^{#1}(Ω)}
\newcommand{\Wom}[2]{W^{#1,#2}(Ω)}
\newcommand{\norm}[2][]{\|#2\|_{#1}}
\newcommand{\normm}[2]{\|#2\|_{#1}}

\newcommand{\io}{\int_{Ω}}
\newcommand{\kl}[1]{\left(#1\right)}
\newcommand{\f}[2]{\frac{#1}{#2}}
\newcommand{\ddt}{\frac{\mathrm{d}}{\mathrm{d}t}}

\newcommand{\leb}[2][\Omega]{{L^{#2}(#1)}}
\newcommand{\sob}[3][\Omega]{{W^{#2, #3}(#1)}}
\newcommand{\con}[2][\Ombar]{{C^{#2}(#1)}}

\newcommand{\hp}{\hphantom}
\newcommand{\pe}{\mathrel{\hp{=}}}

\newcommand{\eps}{\varepsilon}

\newcommand{\gt}{>}
\newcommand{\lt}{<}

\newcommand{\defs}{\coloneqq}
\newcommand{\sfed}{\eqqcolon}

\newcommand{\ra}{\rightarrow}

\newcommand{\embed}{\hookrightarrow}
\DeclareMathOperator{\supp}{supp}

\title{Long-term behaviour in a parabolic--elliptic chemotaxis--consumption model}

\usepackage{authblk}

\author[1]{Mario Fuest\footnote{e-mail: fuestm@math.upb.de}}
\author[1,2]{Johannes Lankeit\footnote{e-mail: jlankeit@math.upb.de}}
\author[3]{Masaaki Mizukami\footnote{e-mail: masaaki.mizukami.math@gmail.com}}

\affil[1]{Institut für Mathematik, Universität Paderborn, Warburger Str.~100, 33098 
Paderborn, Germany}  
\affil[2]{Department of Applied Mathematics and Statistics, Comenius University, 
Mlynsk\'a dolina, 84248 Bratislava, Slovakia}
\affil[3]{Department of Mathematics, Tokyo University of Science, 1-3 Kagurazaka, 
Shinjuku-ku, Tokyo 162-8601, Japan}

\begin{document}
\date{}
\maketitle

\begin{abstract}
\noindent\textbf{Abstract.} Global existence and boundedness of classical solutions of the 
chemotaxis--consumption system 
\begin{align*}
  n_t &= \Delta n - \nabla \cdot (n \nabla c), \\
  0   &= \Delta c - nc,                        
\end{align*}
under no-flux boundary conditions for $n$ and Robin-type boundary conditions
\[
 \partial_{\nu} c = (\gamma-c) g 
\]
for $c$ (with $\gamma>0$ and $C^{1+\beta}(\partial\Omega) \ni g > 0$ for some $\beta\in(0,1)$)
are established in bounded domains 
$\Omega\subset\mathbb{R}^{N}$, $N\ge 1$. 
Under a smallness condition on $\gamma$, moreover, we show convergence to the stationary solution.
\\

\end{abstract}

\section{Introduction}

\textbf{Chemotaxis--consumption models.} 
Chemotaxis is the (partially) directed movement of e.g.\ %
bacteria in response to a chemical signal. In the case of, for example, \textit{Bacillus subtilis} 
in water, the signal substance (here:\ oxygen) is consumed, which leads to the system 
\begin{align}\label{consumptionsystem}
  \begin{cases}
    n_t = Δn - χ∇\cdot(n∇c),\\
    c_t = Δc-nc
  \end{cases}
\end{align}
if we denote by $n$ the population density of bacteria and by $c$ the concentration of 
oxygen and where $χ>0$ stands for the chemotactic sensitivity. 

When coupled to Stokes or Navier--Stokes equations, \eqref{consumptionsystem} turns into the 
prototypical form of the chemotaxis-fluid system intensively investigated in the wake of 
\cite{tuval}. 

It is known that \eqref{consumptionsystem} with homogeneous Neumann boundary conditions in bounded 
two-dimensional domains admits global classical solutions \cite{win_CTNS_global_largedata} and in 
three-dimensional bounded domains has global weak solutions that eventually become smooth, 
\cite{taowin_evsmooth_stabil_consumption}. 

Alternatively, global existence of classical solutions in higher dimensional domains can also be 
achieved by requiring smallness of $\norm[\Lom\infty]{c(\cdot,0)}$ (or more precisely: of 
$χ\norm[\Lom\infty]{c(\cdot,0)}$; cf.\ the simple scaling argument detailed in 
\cite[Introduction]{lankeit_wang}), as shown in \cite{tao_bdoxygenconsumption}. 

Concerning
works in a fluid-context 
we mainly direct the interested reader to Sections~%
4.1 and 4.2 
of the survey \cite{BBTW} and the introductions of \cite{cao_lankeit} and \cite{mizukami_ZAMP} and 
the 
references therein.\\

\textbf{Long-term behaviour and the boundary condition.}
The long-term behaviour in \eqref{consumptionsystem} and its variants (for bounded domains and if 
considered with homogeneous Neumann boundary conditions) in all known cases can be summarized as: 
`Convergence to a constant state'.
Indeed, in \cite{taowin_evsmooth_stabil_consumption} it was shown for \eqref{consumptionsystem} 
that 
for solutions $(n,c)$ of \eqref{consumptionsystem}, $n$ converges to the spatial average 
$\f1{|\Om|}\io n_0$ of the initial data and the second component tends to zero as $t\to 
\infty$. 

For the fluid-coupled system, solutions in two-dimensional domains 
(\cite{win_arma,zhang_li_decay,jiang_wu_zheng,fan_zhao}), classical small-data solutions in 
three-dimensional domains (\cite{cao_lankeit,chae_kang_lee,xia_ye,tan_zhou}), and also all `eventual 
energy solutions' (\cite{win_transAMS}) converge to the same state.

Also if the diffusion is of nonlinear type (\cite{fan_jin}),
the sensitivity function is of more general form (\cite{lswx})
or the signal is consumed only indirectly (\cite{FuestAnalysisChemotaxisModel2019}),
analogous results hold.
(In the former two cases, we refer to \cite{difrancesco_lorz_markowich,win_CalcVarPDE,win_rotationalfluxes} for the corresponding situation with fluid.)

In the presence of logistic source terms modelling population growth (i.e.\ the first equation 
reading $n_t=Δn - χ\nabla\cdot(n∇c) + κn-μn^2$ for some $κ>0$, $μ>0$), the large-time limit for $n$ 
becomes $\f{κ}{μ}$ (see \cite{lankeit_wang}), but remains a constant. 
(Similarly in the fluid-setting, \cite{lankeit_m3as}.) Also in related systems modelling food-supported proliferation, convergence to constants occurs,  \cite{win_food_supported}.

Constants as long-term limit, however, do not seem to be a good fit for the real behaviour of 
\textit{Bacillus subtilis} in water drops, where structure formation has been observed 
experimentally (cf.\ the corresponding discussion in \cite{braukhoff_lankeit} and the experiments in 
\cite{dombrowskietal} and \cite{tuval}). 

In their search for a culprit for this discrepancy, but also with an eye towards the realism and 
appropriateness of the modelling of the behaviour 
of oxygen at the air--water interface, the authors of \cite{braukhoff_lankeit} (following 
\cite{braukhoff}) suggested to replace the usual homogeneous Neumann boundary conditions $∂_{ν}c =0$ 
by 
\begin{equation}\label{realistic-bc}
 ∂_{ν}c = (γ-c)g \qquad \text{on } ∂\Omega,
\end{equation}
so that the assumption of total insulation is replaced by a description of the dissolution of 
gasses in water in accordance with Henry's law (\cite[Section~5.3, page~144]{atkins}). Herein, $γ$ 
indicates the maximal saturation of oxygen in the water and the its influx is 
proportional to the difference between this maximal concentration and the current one. The function 
$g$, which is nonnegative and sufficiently regular on $∂\Om$, can be used to incorporate a 
difference between the air--water and water--ground interfaces. Along the latter, oxygen exchange 
should not happen (or, as an approximation, be extremely slow only, which corresponds to small 
values of $g$).

In \cite{braukhoff_lankeit} it was shown that the stationary system
\begin{align}\label{prob:stationary}
  \begin{cases}
    0   = \Delta n_\infty - \chi \nabla \cdot (n_\infty \nabla c_\infty) & \text{in $\Omega$}, \\
    0   = \Delta c_\infty - n_\infty c_\infty                    & \text{in $\Omega$}, \\
    \partial_\nu n_\infty = n_\infty \partial_\nu c_\infty               & \text{on $\partial \Omega$}, \\
    \partial_\nu c_\infty = (\gamma - c_\infty) g                  & \text{on $\partial \Omega$}
  \end{cases}
\end{align}
(treated in \cite{braukhoff_lankeit} for $χ=1$) has a unique solution for any prescribed positive 
bacterial mass $\io n_{∞}=m>0$, and that, moreover, this solution is nonconstant. 

It was, however, left open whether this stationary solution actually appears as a large-time limit 
of solutions to \eqref{consumptionsystem} with no-flux boundary conditions for $n$ and \eqref{realistic-bc} for $c$.

Very recently, a chemotaxis--Stokes system with Robin boundary conditions generalizing \eqref{realistic-bc} has been studied in \cite{tian_xiang}. For sufficiently strong porous medium type diffusion, weak solutions have been proven to exist globally and to converge to the stationary solution. The result on the long-term behaviour, however, is restricted to the case of no oxygen influx ($γ=0$ in \eqref{realistic-bc}), so that the limit state there, once again, is spatially homogeneous. 

In a one-dimensional setting, a related chemotaxis-consumption system has been studied in 
\cite{knosalla_global} with inhomogeneous Neumann or Dirichlet boundary conditions for the oxygen, 
and steady-states (concerning their existence and uniqueness see \cite{knosalla_nadzieja}) have 
been identified as the large-time limit of solutions (\cite{knosalla_asymptotic}). To the best of 
our knowledge, also for this system the higher-dimensional case is still open. 
Also in \cite{preprint_zhaoyin} (dealing with 
global existence of solutions of a chemotaxis--fluid system in the domain $ℝ^2\times(0,1)$), 
boundary conditions different from homogeneous Neumann conditions are imposed on $c$ on part of the 
boundary, in this case inhomogeneous Dirichlet conditions. For the relation between inhomogeneous 
Dirichlet and no-flux conditions (the former can be regarded as a limiting case of the 
latter), refer to \cite{braukhoff_lankeit}.\\

\textbf{Parabolic--elliptic simplifications of chemotaxis models.} The probably most-studied 
chemotaxis model is the `classical Keller--Segel system' (\cite{KS_70})
\begin{align}\label{classicalKS}
  \begin{cases}
    n_t = Δn - ∇\cdot(n∇c),\\
    τc_t = Δc - c +n, 
  \end{cases}
\end{align}
where (in contrast to \eqref{consumptionsystem}) the signal substance is produced by the 
studied population. 
This system is commonly investigated in either the `fully parabolic' variant 
($τ=1$) or a `parabolic-elliptic' simplification ($τ=0$), which in some sense turns the system 
\eqref{classicalKS} into a single scalar parabolic equation, albeit with a spatially nonlocal term. 

Accessible to additional tools, it was indeed a parabolic--elliptic setting, in which 
many important results, 
including the striking first detection of blow-up (\cite{jaeger_luckhaus}), 
were achieved first (see also, e.g., \cite{nagai95} or \cite{HV_singularitypatterns}). 

For \eqref{consumptionsystem} with homogeneous Neumann conditions
an elliptic simplification of the second equation to $0=Δc-nc$ is 
not possible: every nonnegative solution satisfies $c\equiv 0$ (and the first equation, 
accordingly, 
turns into the much less interesting heat equation). This is different if \eqref{realistic-bc} is 
imposed. In light of the aforementioned modelling considerations supporting this particular choice of boundary conditions, we therefore suggest to view the following as the `correct' parabolic--elliptic variant of 
\eqref{consumptionsystem} (with $\chi = 1$, see also Remark~\ref{rm:chi} below)
  \begin{align} \label{our_system}
    \begin{cases}
      n_t = \Delta n - \nabla \cdot (n \nabla c) & \text{in $\Omega \times (0, \infty)$}, \\
      0   = \Delta c - nc                        & \text{in $\Omega \times (0, \infty)$}, \\
      \partial_\nu n = n \partial_\nu c          & \text{on $\partial \Omega \times (0, \infty)$}, 
\\
      \partial_\nu c = (\gamma - c) g            & \text{on $\partial \Omega \times (0, \infty)$}, 
\\
      n(\cdot, 0) = n_0                          & \text{in $\Omega$}.
    \end{cases}
  \end{align}
  Another way to ensure solvability notwithstanding ellipticity of the second equation, while keeping homogeneous Neumann boundary conditions, is the addition of a source term to said equation. Such a nutrient-taxis system has been studied in \cite{taowin_nutrienttaxis}, the conclusion being global existence of solutions and their convergence to a constant state, if the source becomes homogeneous for large times in a suitable sense. 
From a modelling perspective, using the second equation in an elliptic form corresponds to the 
assumption that oxygen diffuses much faster than bacteria, which often seems justifiable.

Concerning the domain, we will always assume that 
\begin{equation}\label{cond:Om}
  N\in\mathbb{N},\qquad \Om\subset ℝ^N \text{ is a bounded domain with smooth boundary} \tag{$\Om$}.
\end{equation}

For this system \eqref{our_system} we firstly ensure global existence of classical solutions in
\begin{thm}\label{th:global_ex}
 Assume \eqref{cond:Om}.
 Let $\con{1+\beta} \ni g \gt 0$ for some $\beta \in (0, 1)$, $\gamma \gt 0$ and $\con0 \ni n_0 \gt 0$.
 Then there exists a unique tuple  
 \begin{align}\label{reg_nc}
      (n,c)
  \in C^0(\Ombar\times[0,\infty))\cap C^{1, 0}(\Ombar\times(0, \infty)) \cap C^{2,1}(\Om\times(0,\infty))
  \times C^{2,0}(\Ombar\times(0,\infty))
 \end{align}
 solving \eqref{our_system} classically.
 In addition, these functions are nonnegative
 and $(n, c, \nabla c)$ is bounded in $(L^\infty(\Omega \times (0, \infty))^{2+N}$.
\end{thm}

Secondly, we turn our attention to their long-term behaviour and answer the question whether the 
stationary solution found in \cite{braukhoff_lankeit} actually plays the role of a large-time limit 
in the affirmative, at least under a smallness condition on one of the system parameters, but 
without any restriction on the size of the initial data. 

\begin{thm} \label{th:conv}
  Assume~\eqref{cond:Om}.
  \begin{enumerate}
    \item[(i)]
      For any $m, C_g \gt 0$, there is $\gamma_0 \gt 0$ such that if
      \begin{align}\label{eq:conv:cond_n0}
        n_0 \in \con0 \text{ with } 0 < n_0  \text{ and } \io n_0 \le m
      \end{align}
and 
      \begin{align}\label{eq:conv:cond_g_gamma}
        g \in \con{1+\beta} \text { with } \|g\|_{\leb\infty} \le C_g
        \quad \text{and} \quad \gamma \in(0, \gamma_0),
      \end{align}
      then the solution $(n, c)$ given by Theorem~\ref{th:global_ex} fulfills
      \begin{align*}
        n(\cdot, t) \ra n_\infty, \quad
        c(\cdot, t) \ra c_\infty
        \quad \text{and} \quad
        \nabla c(\cdot, t) \ra \nabla c_\infty
        \qquad\text{uniformly as $t \ra \infty$},
      \end{align*}
      where $(u_\infty, v_\infty)$ denotes the unique solution of~\eqref{prob:stationary}
      with $\io n_\infty = \io n_0$.

    \item[(ii)]
      For any $\gamma_0 \gt 0$, there are $m, C_g \gt 0$
      such that if \eqref{eq:conv:cond_n0} and \eqref{eq:conv:cond_g_gamma} are fulfilled,
      then the conclusion from part~(i) holds true.
  \end{enumerate}
\end{thm}

\begin{remark}\label{rm:chi}
  Let $\chi \gt 0$.
  If $(n, c)$ solves \eqref{our_system}, then $(n, \chi c)$ solves
  the system obtained by replacing the first equation in \eqref{our_system} by $n_t = \Delta n - \chi \nabla \cdot (n \nabla c)$
  if one also replaces $\gamma$ by $\chi \gamma$
    (and the boundary condition $\partial_\nu n = n \partial_\nu c$ by $\partial_\nu n = \chi n \partial_\nu c$).

  Thus, Theorem~\ref{th:global_ex} provides global, bounded solution also to the latter system
  and Theorem~\ref{th:conv} holds for that system as well,
  provided the condition $\gamma \le \gamma_0$ in \eqref{eq:conv:cond_g_gamma} is replaced by $\chi \gamma \le \gamma_0$.
\end{remark}

\paragraph{Plan of the paper.}
A popular approach toward solvability of chemotaxis systems (see e.g.\ Lemma~3.1 of the survey \cite{BBTW}) consists in rewriting the system as fixed point problem for mild solutions via Duhamel's formula and applying Banach's fixed point theorem on short time intervals, and later inferring global existence from suitable bounds and an extensibility criterion (like (3.3) of \cite{BBTW}). However, the change of the boundary conditions away from homogeneous Neumann boundary conditions renders some of the semigroup estimates on which even the preparations of the first invocation of Banach's theorem rely inapplicable. We will instead use the Leray--Schauder theorem to directly obtain classical solutions of \eqref{our_system} and spend a large part of this article (up to Lemma \ref{lem:boundedness-n-in-Lp}) on the preparation of appropriate a priori bounds. At the same time, the resulting estimates play an essential role also for boundedness and uniform convergence of the solution. (As to these estimates, Lemmata \ref{lem:existence-c(t)}--\ref{lem:ex:c-u} and Lemmata \ref{lem:ex:n}--\ref{lem:n-c1beta} will focus on the subproblems for $c$ and $n$, respectively.) 

The most crucial estimate will arise from a study of the evolution of $\io n^p$ in Lemma~\ref{lem:boundedness-n-in-Lp} 
. As it is time-independent and already enters the Leray--Schauder reasoning (see Lemma~\ref{lem:locex_c1betadata}) of the existence proof, we do not need the usual split into local and global existence results. We will, however, firstly argue for slightly more regular initial data only (Lemma~\ref{lem:locex_c1betadata}) and remove this restriction later in Lemma~\ref{lem:existence_generalinitialdata}. As preparation for this, and also as essential ingredient for the uniqueness proof, the lemmata in between will deal with a Grönwall type inequality.  
A final global boundedness result (for Hölder norms) and the proof of Theorem~\ref{th:global_ex}, by then only consisting of collecting the right previous results, will be given in Section~\ref{sec:globbd-thm-gex}.

The proof of the large-time behaviour, and thus of Theorem~\ref{th:conv}, will be given in Section~\ref{sec:largetime}. Apart from the bounds previously derived, the proof heavily relies on energy inequalities for 
\[
  \int_\Omega (n(\cdot,t) - n_\infty)^2 
\quad  \mbox{and} \quad 
  \int_\Omega |\nabla (c(\cdot ,t) - c_\infty)|^2 + K \int_\Omega (c(\cdot,t) - c_\infty)^2 
\]
with some $K>0$, where $(n_\infty,c_\infty)$ is the nonconstant steady state of \eqref{our_system}. The non-vanishing boundary terms appearing in these calculations due to the Robin boundary condition in \eqref{our_system} can be dealt with by a suitable interpolation inequality (see Appendix~\ref{sec:interpolation}), variants of which will also have been used in the proofs of Lemma~\ref{lem:boundedness-n-in-Lp} and Lemma~\ref{lem:gronwall}.

\section{Local and global existence}
The existence proof will rely on Schauder's fixed point theorem.
To that end we prepare certain estimates for both subproblems in~\eqref{our_system}.

\subsection{A priori estimates for \texorpdfstring{$c$}{c}}
We start by collecting properties of (solutions to)
\begin{align}\label{ceq}
 \begin{cases}
  0 =Δc - nc & \text{in } \Om,\\
  ∂_{ν}c =(γ-c)g & \text{on } ∂\Om
 \end{cases}
\end{align} 
for given $\gamma, g, n$.
Later on, for each $t \ge 0$, $n$ will be the solution to the other subproblem in~\eqref{our_system} at time $t$.

\begin{lemma}\label{lem:existence-c(t)}
Assume \eqref{cond:Om}. Let $γ\gt 0$, $\beta \gt 0$, $C^{1+β}(∂Ω) \ni g \gt 0$ and $n\in C^{β}(\Ombar)$ with $0 \not\equiv n \ge 0$.
Then there is a unique solution $c\in C^{2+β}(\Ombar)$ of \eqref{ceq}.
This solution satisfies 
\begin{equation}\label{c-positive}
0 < c < γ\qquad \text{ in } \Ombar.
\end{equation}
\end{lemma}
\begin{proof}
 Existence and uniqueness follow from \cite[Corollary~4.41]{lieberman_ellipticbook},
 while regularity is provided by \cite[Theorem~4.40]{lieberman_ellipticbook}.
 The estimates in $\Ombar$ can be shown similarly as in \cite[Lemmata~3.2 and 3.3]{braukhoff_lankeit}.
 Nonetheless, we give the basic idea here:
 If $c$ were constant, then $c \equiv \gamma$ due to the boundary condition and positivity of $g$,
 but combined with the differential equation in \eqref{ceq} this contradicts the assumption that $n \not\equiv 0$.
 Thus, the maximum principle \cite[Lemma~3.5]{GT} provides us with some $x_0 \in \partial \Omega$
 with $c(x_0) \lt c(x)$ for all $x \in \Ombar$.
 If $c(x_0) \le 0$, Hopf's boundary point lemma \cite[Lemma~3.4]{GT} would assert $(\gamma - c(x_0)) g(x_0) = \partial_\nu c(x_0) \lt 0$,
 which contradicts $\gamma g(x_0) \gt 0$.
 Again relying on the maximum principle,
 we moreover obtain $x_1 \in \partial \Omega$ satisfying $c(x_1) \gt c(x) \ge 0$ for all $x \in \Ombar$.
 Hence Hopf's boundary point lemma further asserts $(\gamma - c(x_1)) g(x_1) \gt 0$;
 that is, $c(x_1) \lt \gamma$.
\end{proof}

We next collect the following elliptic Schauder estimate. 
\begin{lemma}\label{lem:bounds-c(t)}
Assume \eqref{cond:Om} and let $β\in(0,1)$. For every $Λ>0$, there is $C>0$ such that whenever $g,γ,n$ satisfy the conditions from Lemma
\ref{lem:existence-c(t)} and 
\[
\norm[C^{β}(\Ombar)]{n}\le Λ, \quad \norm[C^{1+β}(∂Ω)]{g}\le Λ
\]
we have
\[
\norm[C^{2+β}{(\Ombar)})]{c}\le Cγ
\]
for the solution $c$ of \eqref{ceq} given by Lemma~\ref{lem:existence-c(t)}.
\end{lemma}
\begin{proof}
According to \cite[Theorem 2.26]{lieberman_ellipticbook}, there is $C_1=C_1(Λ,Ω)>0$ such that 
\[
\norm[C^{2+β}(\Ombar)]{c} \le C_1 ( \norm[\Lom{∞}]{c} + \norm[C^{1+β}(\partial \Omega)]{gγ} )
\]
for every solution of \eqref{ceq} with $n,u,g$ as indicated in the lemma. The estimate finally
follows from \eqref{c-positive}.
\end{proof}
Any application of Lemma~\ref{lem:bounds-c(t)} presumes Hölder bounds for $n$. If only $\norm[\Lom p]{n}$ is known to be bounded instead, the following lemma yields $W^{2,p}$-regularity of the solution to \eqref{ceq}. 
\begin{lemma}\label{lem:W2p-estimate-c}
Assume~\eqref{cond:Om}. 
Let $p \in (1, \infty)$ and $γ>0$ as well as $g$ be as in Lemma \ref{lem:existence-c(t)}.
There is $C>0$ such that whenever $n$ is as in Lemma \ref{lem:existence-c(t)}, then
the solution $c$ of \eqref{ceq} satisfies 
\[
\norm[\Wom2p]{c}\le C \kl{1+\norm[\Lom p]{n}}.
\]
\end{lemma}
\begin{proof}
We employ \cite[Theorem~1.19.1]{friedman}
to find $C>0$ so that for every $n,c$ as in the lemma 
\[
\norm[\Wom2p]{c} \le C\kl{\norm[\Lom p]{nc}+\norm[\Lom p]{c}}.
\]
We furthermore use that according to Lemma \ref{lem:existence-c(t)}, $0\le c\le γ$ in $\Omega$.
\end{proof}

The following lemma will be used to show H\"older regularity of solutions to 
\eqref{ceq} with respect to time $t$. 

\begin{lemma}\label{lem:regularity:c(t)}
Assume \eqref{cond:Om}. Let $γ>0$, $\beta \in (0, 1)$ and $C^{1+β}(∂Ω) \ni g \gt 0$.
For every $Λ>0$ there is $C>0$
such that whenever $n_1,n_2 \in \con\beta$ are nonnegative but positive at some point,
$c_1,c_2$ denote the corresponding solutions to \eqref{ceq} given by Lemma~\ref{lem:existence-c(t)}
and 
\[
\norm[C^{β}(\Om)]{n_1}\le Λ, \quad \norm[\Lom{N}]{n_2} \le Λ, 
\]
then 
\[
\norm[C^{2+β}(\Ombar)]{c_1-c_2}\le C \kl{\norm[C^{β}(\Ombar)]{n_1-n_2}}
\] 
\end{lemma}
\begin{proof}
The difference $\tilde{c}\defs c_1-c_2$ satisfies 
\[
  \begin{cases}
    (Δ-n_1)\tilde{c} = f           & \text{in } Ω,\\
    ∂_{ν}\tilde{c} = -\tilde{c} g  & \text{on } ∂Ω,
  \end{cases}
\]
where $f=(n_1-n_2)c_2$. 
Then \cite[Theorem~2.26]{lieberman_ellipticbook} provides $C_1(\Lambda)>0$ 
such that 
\[
\norm[C^{2+β}(\Ombar)]{\tilde c}\le C_1(\Lambda)\kl{\norm[C^{β}(\Ombar)]{f} + \norm[\Lom{∞}]{\tilde c}}.
\]
By \cite[Theorem~1.9]{lieberman_ellipticbook},
thanks to our assumption that $g \gt 0$,
we moreover have $C_2>0$ 
such that 
\[
\tilde c \le C_2 \norm[\Lom N]{(n_1-n_2)c_2} \quad\text{ and that } \quad
-\tilde c \le C_2 \norm[\Lom N]{(n_2-n_1)c_1}.
\]
By combining this with Lemma~\ref{lem:existence-c(t)}, we indeed obtain the statement.
\end{proof}

We finally establish that for all smooth functions $n$ now depending on time in addition to space, 
we can find a unique function $c$ 
which solves \eqref{ceq} in the classical sense for every time $t$. 

\begin{lemma}\label{lem:ex:c-u}
  Assume \eqref{cond:Om} and let $\gamma \gt 0$, $\beta \in (0, 1)$, $C^{1+β}(\partial \Omega) \ni g \gt 0$
  and $T \gt 0$.  For any $\Lambda \gt 0$ there is $C \gt 0$ such that the following holds: For all $n\in C^{β,\f{β}2}(\Ombar\times[0,T])$
  with $\norm[C^{β,\f{β}2}(\Ombar\times{[0,T]})]{n}\le Λ$ there is a unique function
  $c\in C^{\f{β}4}([0,T];C^{2+\f{β}2}(\Ombar))$ 
  that solves \eqref{ceq} for every $t\in[0,T]$.
  Additionally, this function satisfies $\norm[C^{\f{β}4}({[0,T]};C^{2+\f{β}2}(\Ombar))]{c}\le C$.
\end{lemma}
\begin{proof}
  For each $t\in [0,T]$, Lemma \ref{lem:existence-c(t)} defines $c(t)$.
  The regularity of this function  with respect to the time variable results from Lemma \ref{lem:regularity:c(t)}
  if applied with $c_1=c(\cdot,t_1)$ and $c_2=c(\cdot,t_2)$ for $t_1,t_2\in[0,T]$ due to $n\in C^{\f{β}4}([0,T];C^{\f{β}2}(\Ombar))$.
\end{proof}

\subsection{A priori estimates for \texorpdfstring{$n$}{n}}
In this subsection we establish a priori estimates for solutions $n$ 
of 
\begin{align}\label{neq}
  \begin{cases}
    n_t = \Delta n - \nabla \cdot (n \nabla c) & \text{in $\Omega \times (0, T)$}, \\
    \partial_\nu n = n \partial_\nu c & \text{on $\partial \Omega \times (0, T)$}, \\
    n(\cdot, 0) = n_0 & \text{in $\Omega$}
  \end{cases}
\end{align}
for given $c$ and $n_0$.  
We first verify existence of solutions to \eqref{neq}. Here and below, we understand the term `weak solution' in the sense of \cite[p.\ 136]{lieberman_parabolicbook}. 
\begin{lemma}\label{lem:ex:n}
Assume \eqref{cond:Om}. Let $β\in(0,1)$, $T\in(0,∞)$, $∇c\in C^{β,\f{β}2}(\Ombar\times[0,T])$, and let $n_0\in C^{1+β}(\Ombar)$ satisfy $∂_{ν}n_0=n_0∂_{ν}c(\cdot,0)$ on $∂Ω$.
Then there is a (weak and) classical solution 
\begin{align}\label{eq:ex:n:reg}
  n\in X \defs 
  \left\{\,
    \psi \in C^{1,\f{1}2}(\Ombar\times[0,T])\cap C^{2,1}(\Om\times(0,T)) \cap L^1((0, T); \sob21) :
    \psi_t \in L^1(\Omega \times (0, T))
  \,\right\}
\end{align}
of \eqref{neq}.
This solution 
is unique in $C^{1,\f12}(\Ombar\times[0,T])$.

Moreover, for every $Λ>0$ there is $C>0$ such that whenever the conditions 
\[
 β\le Λ, \quad \normm{C^{β,\f{β}2}(\Ombar\times[0,T])}{∇c\cdot ν}\le Λ, \quad \normm{C^{β,\f{β}2}(\Ombar\times[0,T])}{∇c}\le Λ, \quad \norm[C^{1+β}(\Ombar)]{n_0}\le Λ
\]
are fulfilled, then the estimate 
\begin{equation}\label{intermediateschauderestimate}
 \normm{C^{1+β,\f{1+β}2}(\Ombar\times[0,T])}{n} \le C
\end{equation}
holds true.
\end{lemma}
\begin{proof}
Existence of a (weak) solution with $C^{1+β}$-regularity and the uniqueness assertion result from \cite[Theorem~6.46]{lieberman_parabolicbook}; the same item also yields the estimate \eqref{intermediateschauderestimate}. 
That $n_t$ and $Δn$ belong to $L^p(\Ombar\times(τ,T))$ for $p\in[1,\f1{1-β})$ and every $τ\in(0,T)$ is a consequence of \cite[Theorem~7.20]{lieberman_parabolicbook}.
In order to obtain $n\in C^{2,1}(\Om\times(0,T))$, we pick some $φ\in C_c^{\infty}(\Om\times(0,T))$ and observe that $u=nφ$ solves the homogeneous Dirichlet problem with zero initial data for $u_t=Δu-∇\cdot(au)+f$ with $a=∇c\in C^{1+β,\f{β}2}(\Ombar\times[0,T])$ and $f=-nΔφ-2∇n∇φ+n∇c∇φ+nφ_t\in C^{β,\f{β}2}(\Ombar\times[0,T])$. By \cite[Theorem~III.3.4]{LSU}, $nφ$ therefore has to coincide with the solution provided by \cite[Theorem~IV.5.2]{LSU}, which belongs to $C^{2+β,1+\f{β}2}(\Ombar\times[0,T])$.
\end{proof}

  As Lemma~\ref{lem:ex:n} does not assert sufficient regularity of $n_t$ to justify calculations of the form
$\ddt \io n^p = p \io n^{p-1} n_t$ (almost everywhere) for $p \in [1, \infty)$, we need to replace typical testing procedures
by certain weak counterparts thereof. These are provided by the following quite general lemma.

\begin{lemma}\label{lm:weak_testing}
  Assume \eqref{cond:Om}.   
  Let  $T \in (0, \infty)$, $n_0 \in \con0$
  and $f \in C^0(\Ombar\times[0,T];ℝ^N)$.

  Suppose that $n \in L^1((0, T); \sob21)\cap C^0(\Ombar\times(0,T))$ with $n_t \in L^1(\Omega \times (0, T))$ is a weak solution (according to \cite[p.~136]{lieberman_parabolicbook}) of
  \begin{align*}
    \begin{cases}
      n_t = \nabla\cdot(\nabla n+f) & \text{in $\Omega \times (0, T)$}, \\
      \partial_\nu n = - f\cdot ν & \text{on $\partial \Omega \times (0, T)$}, \\
      n(\cdot, 0) = n_0 & \text{in $\Omega$}.
    \end{cases}
  \end{align*}

  Then
  \begin{align*}
      \frac{\psi(t)}{p} \io n^p(\cdot, t)
    = \frac{\psi(0)}{p} \io n_0^p
      + \f1p\int_0^t \left( \psi' \io n^p \right)
      - \int_0^t \left( \psi \io (\nabla n + f) \cdot \nabla n^{p-1} \right)
  \end{align*}
  for all $t \in (0, T)$, $p \in \N$ and $\psi \in C^1([0, T])$.
\end{lemma}
\begin{proof}
 We fix $t \in (0, T)$, $p \in \N$ and $\psi \in C^1([0, T])$.
  The asserted regularity of $n_t$ and $\psi$ implies that by means of approximaion arguments, $\varphi \defs n^{p-1} \psi$ can be used as test function in the defining integral identity \cite[p.~136]{lieberman_parabolicbook} for weak solutions, that is, 
  \begin{align}\label{eq:weakform}
      \left[ \io n n^{p-1} \psi \right]_0^t
      - \int_0^t \io n (n^{p-1} \psi)_t
    = - \int_0^t \left( \psi \io (\nabla n + f) \cdot \nabla n^{p-1} \right).
  \end{align}
  For the second term therein we have 
   \begin{align*}
       \int_0^t \io n (n^{p-1} \psi)_t
     &= \int_0^t\io n^pψ'+\f{p-1}p\int_0^t\io (n^p)_tψ\\
     &= \int_0^t\io n^pψ' -\f{p-1}p\int_0^t\io n^pψ' + \f{p-1}p \left[ψ\io n^p\right]_0^t,
   \end{align*}
  which, when inserted into \eqref{eq:weakform}, implies the statement.
\end{proof} 

Since the first equation in \eqref{neq} is of divergence form, 
we have the mass conservation law of the solution $n$. 
\begin{lemma}\label{lm:mass_conservation}
  Assume \eqref{cond:Om} and let $n_0 \in \con0$ as well as $c\in C^{2+β,\f{β}2}(\Ombar\times[0,T])$ with $T>0$.
  Every solution $n\in C^0(\Ombar\times[0,T])\cap C^{2,1}(\Om\times(0,T))$ of \eqref{neq} satisfies 
  \begin{align*}
    \io n(\cdot, t) = \io n_0
    \qquad \text{for all $t \in [0, T)$}.
  \end{align*}
\end{lemma}
\begin{proof}
  Choose $f \defs -n \nabla c$, $\psi \equiv 1$ and $p \defs 1$ in Lemma~\ref{lm:weak_testing}.
\end{proof}
%

The following two lemmata guarantee the smoothness of $n$ up to the spatial boundary.
\begin{lemma}\label{lem:hoelderbound:n}
Assume \eqref{cond:Om} and let $T>0$.
\begin{enumerate}
  \item[(i)]
    Given $Λ>0$ and $\beta \in (0, 1)$ there are $C>0$ and $\beta' \in (0, 1)$ such that whenever 
    \[
      \norm[C^{β}(\Ombar)]{n_0}\le Λ, \quad \norm[L^{∞}(\Ombar\times(0,T))]{n∇c}\le Λ
    \]
    and $n\in C^0([0,T];\Lom2)\cap L^2((0,T);\Wom12)\cap L^\infty(\Om\times(0,T))$ weakly solves \eqref{neq}, then 
    \[
    \normm{C^{β',\f{β'}2}(\Ombar\times[0,T])}{n}\le C.
    \]

  \item[(ii)]
    Moreover, given $Λ>0$ and $ε>0$ there are $C>0$ and $\beta' \in (0, 1)$ such that whenever 
    \[
    \norm[L^{∞}(\Ombar\times(0,T))]{n∇c}\le Λ
    \]
    and $n$ solves \eqref{neq} weakly, then 
    \[
    \normm{C^{β',\f{β'}2}(\Ombar\times[ε,T])}{n}\le C.
    \]
\end{enumerate}
\end{lemma}
\begin{proof}
 See \cite[Theorem 4]{benedetto_Local}. 
\end{proof}

\begin{lemma}\label{lem:n-c1beta}
Assume \eqref{cond:Om} and let $\beta \in (0, 1)$.
    For every $Λ>0$, $T>0$ and $ε>0$, there are $C>0$ and $β'\in(0,1)$
    such that whenever 
        \begin{align*}
      n \in L^\infty(\Omega \times (0, T)) \cap C^0([0, T]; L^2(\Omega))\cap L^2((0,T);\Wom12)
    \end{align*}
    is a weak solution of \eqref{neq}
  for some $n_0\in C^0(\Ombar)$, $∇c \in C^{\beta, \frac{\beta}{2}}(\Ombar \times [0, T])$, 
    \begin{align*}
      \|n\|_{L^\infty(\Omega \times (\f{ε}2, T))} \le \Lambda 
      \quad \text{and} \quad
      \|\nabla c\|_{C^{\beta, \frac{\beta}{2}}(\Ombar \times [\f{ε}2, T])} \le \Lambda,
    \end{align*}
    then we have 
    \[
     \normm{C^{1+β',\f{1+β'}2}(\Ombar\times[ε,T])}{n}\le C.
    \]
\end{lemma}
\begin{proof}
  For this temporally localized statement one can apply \cite[Theorem~1.1]{lieberman_Holder} to $ζn$
  for a cutoff function $ζ\in C_c^{∞}([0,∞))$ with $ζ(t)=0$ for $t\in[0,\f{ε}2)$ and $ζ(t)=1$ for every $t\in[ε,T]$.
\end{proof}

\subsection{Existence and boundedness for sufficiently smooth initial data}

The core of the boundedness proof lies in the following lemma. It is based on a study of the temporal evolution of $\io n^p$, during which we will employ ellipticity of the equation for $c$. The boundary integrals originating from \eqref{realistic-bc} will be dealt with by means of trace embeddings (for the corresponding interpolation inequality refer to Appendix~\ref{sec:interpolation}). In the second half of the proof, this study transitions into a Moser-type iteration yielding boundedness in $L^{∞}$.
In preparation for the Leray--Schauder fixed point theorem (which requires a priori estimates for solutions of $n = \sigma \Phi(n)$, $Φ$ being the solution operator for \eqref{neq}, for every $\sigma\in[0,1]$), we introduce a parameter $σ$ into the system.
\begin{lemma}\label{lem:boundedness-n-in-Lp}
\label{lem:boundedness-n-in-Linfty}
Assume \eqref{cond:Om}, let $\gamma \gt 0$, $\con0 \ni g \gt 0$
and $Λ>0$.
Then there is $C>0$ such that whenever $T \in (0, \infty]$, $σ\in[0,1]$ and
\begin{align}\label{reg_nc_t}
  (n,c)\in X \times C^{2,0}(\Ombar\times(0,T)),
\end{align}
with $X$ as in \eqref{eq:ex:n:reg}, 
solves
\begin{align} \label{prob:sigma}
  \begin{cases}
    n_t = Δn - ∇\cdot(n∇c)            & \text{in }\Om\times(0,T),\\
    0 = Δc - σnc                      & \text{in }\Om\times(0,T),\\
    \partial_\nu n = n \partial_\nu c & \text{on }∂\Om\times(0,T),\\
    \partial_\nu c = (\gamma - c)g    & \text{on }∂\Om\times(0,T),\\
    n(\cdot,0) = n_0                   & \text{in }\Om
  \end{cases}
\end{align}
(weakly and) classically for some $n_0\in C^0(\Ombar)$ with $\norm[\Lom \infty]{n_0}\le Λ$, then 
\[
 \norm[\Lom \infty]{n(\cdot,t)} \le C \qquad \text{for all } t\in(0,T).
\]
\end{lemma}
\begin{proof}
  We begin by fixing some constants:
  Let $C_1 \defs \gamma \|g\|_{L^\infty(\partial\Omega)}$
  and choose $C_2 \gt 0$ sufficiently large such that $\frac{C_1}{C_2 p} \le \frac{2(p-1)}{p^2}$ for all $p \in [2, \infty)$.
  Fixing an arbitrary $\lambda \in (0, \frac1{N+2})$ and setting $\mu \defs \frac{1-\lambda}{\lambda} \gt 0$,
  we employ Lemma~\ref{lm:trace_thm}(ii) (with $q \defs 1$)
  to obtain $C_3 \gt 0$ such that
  \begin{align}\label{eq:boundedness-n-in-Lp:bdr_est}
    \int_{∂Ω} \varphi^2
    \le \frac1{C_2 p} \io |∇ \varphi|^2 + C_3 p^\mu \left( \io \varphi \right)^2
    \qquad \text{for all $\varphi \in \sob12$ and all $p \in [2, \infty)$}.
  \end{align}
  Moreover, Poincar\'e's inequality provides us with $C_4, C_5 \gt 0$ such that
  \begin{align}\label{eq:boundedness-n-in-Lp:poincare}
          \frac{C_1}{C_2} \io |∇\varphi|^2
    &\ge  C_4 \io \varphi^2 - C_5 \left(\io \varphi \right)^2
    \qquad \text{for all $\varphi \in \sob12$}.
  \end{align}
  Finally, we set $C_6 \defs C_1 C_3 + C_5$ and $C_7 \defs \frac{2C_6}{C_4}$ and fix $T' \in (0, T)$.
  For $p \in \N$,
  Lemma~\ref{lm:weak_testing} (with $f \defs -n \nabla c$ and $\psi(t) \defs \ure^{C_4 t}$ for $t \in [0, T)$)
  assert that
  \begin{align}\label{eq:boundedness-n-in-Lp:first_testing}
          \frac{ \ure^{C_4 T'}}{p} \io n^p(\cdot, T')
    &=    \frac1p \io n_0^p
          + \f{C_4} p \int_0^{T'} \left( \ure^{C_4 t} \io n^p \right)
          - \int_0^{T'} \left( \ure^{C_4 t} \io (\nabla n - n \nabla c) \cdot \nabla n^{p-1} \right)
  \end{align}
  holds.
  
  By replacing $Δc$ by $σnc$ and removing negative terms
  we obtain that
  \begin{align*}
    \pe  -\io (\nabla n - n \nabla c) \cdot \nabla n^{p-1} 
    &=    -(p-1) \io n^{p-2} |∇n|^2 + (p-1) \io n^{p-1} ∇n\cdot ∇c \nn\\
    &=    -\f{4(p-1)}{p^2} \io |∇n^{\f p2}|^2 + \f{p-1}p \io ∇n^p\cdot∇c\nn\\
    &=    -\f{4(p-1)}{p^2} \io |∇n^{\f p2}|^2 -\f{p-1}p \io n^p Δc + \f{p-1}p\int_{∂Ω} n^p (γ-c)g \nn\\
    &=    -\f{4(p-1)}{p^2} \io |∇n^{\f p2}|^2 -\f{p-1}p σ\io n^{p+1}c + \f{p-1}p\int_{∂Ω} n^p (γ-c)g\nn \\
    &\le  - \f{2C_1}{C_2p} \io |∇n^{\f p2}|^2  + C_1 \int_{∂Ω} \left(n^\frac p2\right)^2 \nn\qquad \text{in $(0, T)$ for all $p \in [2, \infty)$}. 
    \end{align*}
 If we additionally make use of \eqref{eq:boundedness-n-in-Lp:bdr_est} and \eqref{eq:boundedness-n-in-Lp:poincare}, we see that 
\begin{align*}
\pe  -\io (\nabla n - n \nabla c) \cdot \nabla n^{p-1}
    &\le  - \f{C_4}p  \io n^p + \left(C_1 C_3 p^\mu + \frac{C_5}{p}\right) \left( \io n^\frac p2 \right)^2
    \qquad \text{in $(0, T)$ for all $p \in [2, \infty)$}.
  \end{align*}
  
  Since $(C_1C_3 p^\mu + \f{C_5}p) \le C_6 p^\mu$ for $p \gt 1$, plugging this into \eqref{eq:boundedness-n-in-Lp:first_testing} yields 
  \begin{align*}
          \frac{ \ure^{C_4 T'}}{p} \io n^p(\cdot, T')
    & \le    \frac1p \io n_0^p
          + C_6 p^\mu \int_0^{T'} \left( \ure^{C_4 t} \left(\io n^{\f p2}\right)^2 \right) \\
    &\le  \frac1p \io n_0^p
          + \frac{C_6}{C_4} p^{\mu} \ure^{C_4 T'} \sup_{t \in (0, T')} \left(\io n^{\f p2}(\cdot, t)\right)^2 
    \qquad \text{for all $p \in \N \setminus \{1\}$}
  \end{align*}
  and hence
  \begin{align*}
          \io n^p(\cdot, T')
    &\le  \max\left\{ 2\io n_0^p, C_7 p^{\mu + 1} \sup_{t \in (0, T')} \left(\io n^{\f p2}(\cdot, t)\right)^2 \right\}
    \qquad \text{for all $p \in \N \setminus \{1\}$}.
  \end{align*}

  As $T' \in (0, T)$ was chosen arbitrary,
  setting $p_j \defs 2^j \in \N$ and $A_j \defs \sup_{t \in (0, T)} \|n(\cdot, t)\|_{\leb {p_j}}$ for $j \in \N_0$,
  we obtain
  \begin{align*}
          A_j
    &\le  \max\left\{ 2^\frac1{p_j} \|n_0\|_{\leb{p_j}}, (C_7 p_j^{\mu + 1})^\frac1{p_j} A_{j-1} \right\}
    \qquad \text{for all $j \ge 1$}.
  \end{align*}
  We introduce the sequence $(B_j)_{j\inℕ}$ defined by 
  \[
   B_0 = Λ\max\left\{1,|\Omega|\right\}, \quad B_j = \max\left\{2 \Lambda\max\left\{|\Om|,1\right\},(C_7p_j^{\mu+1})^{\f1{p_j}} B_{j-1}\right\}\qquad \text{for all } j\ge1.
  \]
  Then by Lemma~\ref{lm:mass_conservation}, $A_0\le B_0$; and because of
    \begin{align*}
        2^\frac1{p_j} \|n_0\|_{\leb{p_j}}
    \le 2 \|n_0\|_{\leb\infty} \max\{|\Omega|, 1\}
    \le 2 \Lambda \max\{|\Omega|, 1\}
    \qquad \text{for all $j \in \N$}
  \end{align*} and due to monotonicity of the right-hand side with respect to $B_{j-1}$, also $A_j\le B_j$ for every $j\in ℕ_0$. Therefore $\|n\|_{L^\infty(\Omega \times (0, T))} = \lim_{j \ra \infty} A_j \le \liminf_{j\to\infty} B_j$. We hence only have to show $\liminf_{j\to\infty} B_j<\infty$ to prove the lemma. 
  
  If there is a sequence $(j_k)_{k \in \N} \subset \N_0$ with $j_k \ra \infty$ for $k \ra \infty$
  and $B_{j_k} \le 2 \Lambda \max\{|\Omega|, 1\}$ for all $k\in ℕ$,
  then also $\liminf_{j \ra \infty} B_{j}\le \liminf_{k\ra\infty} B_{j_k} \le 2 \Lambda \max\{|\Omega|, 1\}$.

  Thus, we may assume that there is $j_0 \in \N$ with $B_j \gt 2 \Lambda \max\{|\Omega|, 1\}$ for all $j \ge j_0$.
   This implies
  \begin{align*}
          B_j
    &\le  (C_7 p_j^{\mu + 1})^\frac1{p_j} B_{j-1}
    \qquad \text{for all $j \ge j_0$}
  \end{align*}
  and hence by induction also
  \begin{align*}
        B_j
    \le \left(\prod_{k = j_0}^{j} (C_7 p_k^{\mu + 1})^\frac1{p_k} \right) B_{j_0}
    =   C_7^{\sum_{k = j_0}^j \frac1{2^k}}
        \cdot 2^{(\mu + 1) \sum_{k = j_0}^j \frac{k}{2^k}}\cdot B_{j_0}
    \qquad \text{for all $j \ge j_0$}
  \end{align*}
  since $p_j = 2^j$ for $j \in \N$.
  This yields the statement
  because of $\sum_{k=1}^\infty \frac{k}{2^k} \lt \infty$ and $B_{j_0} \lt \infty$.
\end{proof}

Thanks to this a priori $L^\infty$  bound of $n$, 
we can assert existence of solutions to \eqref{our_system} for sufficiently smooth initial data. 
\begin{lemma}\label{lem:locex_c1betadata}
  Assume \eqref{cond:Om}, $\gamma \gt 0$ , $\beta \in (0, 1)$ and $C^{1 + \beta}(\partial \Omega) \ni g \gt 0$.
  Let $T \in(0,∞)$ and let $n_0\in C^{1+β}(\Ombar)$ satisfy the compatibility condition $\partial_\nu n_0=n_0\partial_\nu c_0$ on $\partial\Om$, where $c_0$ is the solution of \eqref{ceq} for $n=n_0$.
  There is a solution $(n, c)$ of regularity~\eqref{reg_nc_t}
  to~\eqref{our_system}.
\end{lemma}
 \begin{proof}
We let $β'\in(0,1)$ be as given by Lemma \ref{lem:hoelderbound:n} (i) and abbreviate $X_1\defs C^{β',\f{β'}2}(\Ombar\times[0,T])$. Given $n\in X_1$, we let $Φ(n)\defs \tilde{n}$ be the solution (guaranteed to exist by Lemma \ref{lem:ex:n}) of 
  \[
    \begin{cases}
      \tilde{n}_t = Δ\tilde n - ∇\cdot (\tilde n ∇c) & \text{in $\Omega \times (0, T)$}, \\
      \partial_\nu \tilde n = \tilde n \partial_\nu c & \text{on $\partial\Om \times (0, T)$}, \\
      \tilde n(\cdot, 0) = n_0                       & \text{in $\Omega$},
    \end{cases}
  \]
  where $c\in C^{\f{β'}4}([0,T];C^{2+\f{β'}2}(\Ombar))$ is as given by Lemma~\ref{lem:ex:c-u} for $n$. 
 
  Let $B\subset X_1$ be a bounded set. Then aided by the estimates in  Lemma~\ref{lem:ex:c-u} and Lemma~\ref{lem:ex:n} we can conclude that with some $C_1=C_1(B)>0$ we have 
  \begin{equation}\label{x2bound}
   \norm[X_2]{Φ(n)}\le C_1 
  \end{equation}
for every $n\in B$, where $X_2=C^{1+\f{β'}2,\f{1 + \f{β'}2}2}(\Ombar\times[0,T])$. 

If we assume that a sequence $\{n_j\}_{j\inℕ}$ converges to $n$ in $X_1$, then $c_j\to c$ in $C^{\f{β'}4}([0,T];C^{2+\f{β'}2}(\Ombar))$ 
by Lemma~\ref{lem:regularity:c(t)}, and due to \eqref{x2bound} 
for any  $β''\in(0,\f{β'}2)$ there is a subsequence $\{n_{j_k}\}_{k\inℕ}$ such that $Φ(n_{j_k})\to \hat n$ in $X_3\defs C^{1+β'',\f{1+β''}2}(\Ombar\times[0,T])$, where $\hat n$ solves \eqref{neq} (in the weak sense). By Lemma~\ref{lem:ex:n}, weak solutions to \eqref{neq} are unique in $X_3$, and accordingly $\hat n=Φ(n)$. As the limit object of subsequences is hence uniquely determined, we have that even $Φ(n_j)\to Φ(n)$ in $X_3$ as $j\to \infty$. Since both $X_2$ and $X_3$ are (compactly) embedded into $X_1$, we conclude that $Φ\colon X_1\to X_1$ is continuous and, due to \eqref{x2bound}, compact. 

We now want to procure estimates for every $n\in X_1$ and $σ\in[0,1]$ which are such that $n=σΦ(n)$. 
With $C_3>0$ taken from Lemma \ref{lem:boundedness-n-in-Lp}, we have that for any such $n$ and $σ$ 
\[
 \norm[L^{∞}(\Om\times(0,T))]{n}=σ\norm[L^{∞}(\Om\times(0,T))]{Φ(n)}\le \norm[L^{∞}(\Om\times(0,T))]{Φ(n)} \le C_3.
\]
Lemma \ref{lem:W2p-estimate-c} provides us with $C_4>0$ such that the solution $c$ from Lemma \ref{lem:ex:c-u} (applied to data $σn$) satisfies 
\[
 \norm[L^{∞}(\Om\times(0,T))]{∇c} \le C_4(1+\norm[L^{∞}(\Om\times(0,T)]{σn}))\le C_4(1+C_3),
\]
regardless of the precise choice of $n\in X_1$ or $σ\in[0,1]$. 
Lemma \ref{lem:hoelderbound:n} (i) (with $Λ=\max\{\norm[C^{β'}(\Ombar)]{n_0},C_3C_4(1+C_3)\}$) yields $C_5>0$ with 
 \[
  \norm[X_1]{n} \le C_5
 \]
 for every $n\in X_1$ fulfilling $n=σΦ(n)$ for arbitrary $σ\in[0,1]$. The Leray--Schauder theorem (\cite[Theorem~11.3]{GT}) hence proves existence of a solution $n\in X_1$ with $n=Φ(n)$. At the same time, the above construction (in particular Lemmata \ref{lem:ex:c-u} and \ref{lem:ex:n}) guarantee the asserted regularity. 
\end{proof}

\subsection{Positivity, uniqueness and solvability for less regular initial data}
We next derive positivity of $n$, which will also be used in the proof of the succeeding Gronwall type lemma.  
\begin{lemma}\label{lem:positivity}
Assume \eqref{cond:Om}.  
 Let $T>0$, $γ>0$, $g>0$.
 Then for every $Λ>0$ there is $C>0$ such that whenever $c\in C^{1,0}(\Ombar\times(0,T])$ and $n_0\in C^0(\Ombar)$ are such that 
 \[
  \inf_{\Om} n_0 > \f1{Λ}, \quad
  \norm[L^{∞}(\Om\times(0,T))]{Δc} \le Λ, \quad
  0<c<γ \text{ in } \Ombar\times(0,T)
 \]
 and $n\in C^{2,1}(\Om\times(0,T))\cap C^0(\Ombar\times[0,T))\cap C^{1,0}(\Ombar\times(0,T))$ solves 
 \[
  \begin{cases}
   n_t = Δn-∇\cdot(n∇c)&\text{in } \Om\times(0,T),\\
   ∂_{ν}n = n(γ-c)g& \text{on } ∂\Om\times(0,T),\\
   n(\cdot,0)=n_0&\text{in } \Om,
  \end{cases}
 \]
 then 
 \[
  \inf_{(x,t)\in\Om\times(0,T)} n(x,t) \ge C.
 \]
 \end{lemma}
\begin{proof}
 We let $ε>0$ be so small that $\f1{Λ} \ure^{-Λt} -ε\ure^t>0$ for all $t\in[0,T]$ and introduce 
 \[
  w(x,t)\defs n(x,t)-\f1{Λ}\ure^{-Λt}+ε\ure^t, \qquad (x,t)\in\Ombar\times[0,T],
 \]
claiming that $w\ge 0$ throughout $\Ombar\times[0,T]$. 

If this were not the case, we could find $(x,t)\in \Ombar\times[0,T]$ such that $w(x,t)=0$ and $w(y,s)>0$ for every $y\in\Ombar$, $s\in[0,t)$. From positivity of $w(\cdot,0)$ and continuity of $w$, we could conclude that $t>0$. Moreover, we could observe that $n(x,t)=\f1{Λ}\ure^{-Λt}-ε\ure^t$. Since $x\in \Ombar$ would be a global minimum of $w(\cdot,t)$, either $Δw(x,t)\ge 0$, $∇w(x,t)=0$ (if $x\in \Om$) or $∂_{ν}w(x,t)\le 0$ (if $x\in ∂\Om$).
The latter case would entail that 
\[
 0\ge ∂_{ν}w(x,t)=∂_{ν}n(x,t)=n(x,t)(γ-c(x,t))g(x) = \left(\f1{Λ}\ure^{-Λt}-ε\ure^t\right)(γ-c(x,t))g(x),
\]
which is positive according to the assumptions of $γ,g$ and $c$ and the choice of $ε$, whereas the former case would lead to another contradiction arising from the fact that at $(x,t)$
\begin{align*}
 0&\ge \left(n-\f1{Λ}\ure^{-Λt}+ε\ure^t\right)_t = Δn - ∇n\cdot ∇c -nΔc +\ure^{-Λt}+ε\ure^t \\
 &\ge - nΛ +\ure^{-Λt}+ε\ure^t  = -\ure^{-Λt}+ε\Lambda\ure^t+\ure^{-Λt}+ε\ure^t =ε(\Lambda+1)\ure^t
  \gt 0.
\end{align*}
Finally taking $ε\searrow 0$, we conclude that $n(x,t)\ge \f1{Λ}\ure^{-Λt}$ for every $(x,t)\in\Ombar\times[0,T]$.
\end{proof}
The following lemma plays an important role in the proof of uniqueness of solutions to \eqref{our_system}. 
\begin{lemma} \label{lem:gronwall}
 Assume \eqref{cond:Om}, let $T\in(0,∞)$, $\gamma \gt 0$, $\beta \in (0, 1)$, $C^{1+\beta}(\partial \Omega) \ni g \ge 0$ and $\con 0 \ni n_{0,1} \ge 0$, $\con 0 \ni n_{0,2} \ge 0$.
 For every $Λ>0$ there is $C>0$ such that whenever 
 $(n_1,c_1),(n_2,c_2)\in X \times C^{2,0}(\Ombar \times (0, T))$ with $X$ as in \eqref{eq:ex:n:reg} solve \eqref{our_system} with $n_0=n_{0,1}$ and $n_0=n_{0,2}$, respectively, and fulfill
 \[
  0\le n_1\le Λ
  \quad \text{as well as} \quad
  \frac1\Lambda \le n_2\le Λ
  \qquad \text{in } \Om\times(0,T),
 \]
 then
 \[
  \norm[\Lom2]{n_1(\cdot,t)-n_2(\cdot,t)} \le \ure^{Ct} \norm[\Lom2]{n_{0,1}-n_{0,2}} \qquad \text{for all } t\in[0,T).
 \]
\end{lemma}
\begin{proof}
  By assumption
  \begin{align*}
    \begin{cases}
      (n_1 - n_2)_t = \Delta (n_1 - n_2) - \nabla \cdot( (n_1 - n_2) \nabla c_1) - \nabla \cdot( n_2 \nabla (c_1 - c_2))
        & \text{in $\Omega \times (0, T)$}, \\
      \partial_\nu (n_1 - n_2) = (n_1 - n_2) \partial_\nu c_1 + n_2 \partial_\nu (c_1 - c_2) & \text{on $\partial \Omega \times (0, T)$}, \\
      (n_1 - n_2)(\cdot, 0) = n_{0,1} - n_{0,2} & \text{in $\Omega$}
    \end{cases}
  \end{align*}
  holds in the weak sense, 
  so that 
  Lemma~\ref{lm:weak_testing}
  (applied to $n_1 - n_2$ with $f \defs - (n_1-n_2) \nabla c_1 - n_2 \nabla (c_1 - c_2)$, $\psi \equiv 1$ and $p \defs 2$)
  asserts
  \begin{align} \label{eq:uniqueness:ddt_n1_n2}
    &\pe  \frac12 \io (n_1(\cdot, t) - n_2(\cdot, t))^2
          - \frac12 \io (n_{0,1} - n_{0,2})^2
          + \int_0^t \io |\nabla (n_1 - n_2)|^2 \notag \\
    &=    \int_0^t \io [(n_1 - n_2) \nabla c_1 + n_2 \nabla (c_1 - c_2)] \cdot \nabla (n_1 - n_2) \notag \\
    &=    \frac12 \int_0^t \io \nabla (n_1 - n_2)^2 \cdot \nabla c_1
          + \int_0^t \io n_2 \nabla (n_1 - n_2) \cdot \nabla (c_1 - c_2) \notag \\
    &\sfed I_1(t) + I_2(t) \qquad \text{for all $t \in (0, T)$}.
  \end{align}
  By Lemma~\ref{lem:existence-c(t)} we have $0 \le c_1, c_2 \le \gamma$ in $\Ombar \times [0, T)$.
  Lemma~\ref{lm:trace_thm} (ii) (with $q=2$) provides us with $C_1 \gt 0$ such that 
  \begin{align} \label{eq:uniqueness:i1}
          I_1(t)
    &=    - \frac12 \int_0^t \io (n_1 - n_2)^2 \Delta c_1
          + \int_0^t \int_{\partial \Omega} (n_1 - n_2)^2 (\gamma - c_1) g \notag \\
    &\le  - \frac12 \int_0^t \io (n_1 - n_2)^2 n_1 c_1
          + \gamma\|g\|_{\leb \infty} \int_0^t \int_{\partial \Omega} (n_1 - n_2)^2 \notag \\
    &\le  \frac12 \int_0^t \io |\nabla (n_1 - n_2)|^2
          + C_1 \int_0^t \io (n_1 - n_2)^2\qquad \text{for all } t \in (0,T).
  \end{align}
  As moreover
  \begin{align*}
    0 = \Delta (c_1 - c_2) - (n_1 - n_2) c_1 - n_2 (c_1 - c_2)
    \qquad \text{in $\Omega \times (0, T)$}
  \end{align*}
  and hence 
  \begin{align*}
          \io |\nabla (c_1 - c_2)|^2
    &=    - \int_{\partial \Omega} (c_1 - c_2)^2 g
          - \io c_1 (n_1 - n_2) (c_1 - c_2)
          - \io n_2 (c_1 - c_2)^2 \\
    &\le  0
          + \gamma \io |n_1 - n_2| |c_1 - c_2|
          - \frac1{\Lambda} \io (c_1 - c_2)^2 \\
    &\le  \frac{\gamma^2 \Lambda}{4} \io (n_1 - n_2)^2\qquad \text{in } (0,T)
  \end{align*}
  by an integration by parts and Young's inequality, we may further estimate
  \begin{align} \label{eq:uniqueness:i2}
          I_2 (t) 
    &\le  \frac12 \int_0^t \io |\nabla (n_1 - n_2)|^2
          + \frac{\Lambda^2}{2} \int_0^t \io |\nabla (c_1 - c_2)|^2 \notag \\
    &\le  \frac12 \int_0^t \io |\nabla (n_1 - n_2)|^2
          + C_2 \int_0^t \io (n_1 - n_2)^2
    \qquad \text{for all } t \in (0,T),
  \end{align}
  where we have again used Young's inequality and set $C_2 \defs \frac{\gamma^2 \Lambda^3}{8}$.

  By inserting \eqref{eq:uniqueness:i1} and \eqref{eq:uniqueness:i2} into \eqref{eq:uniqueness:ddt_n1_n2},
  we see that, for all $t \in (0, T)$,
    \begin{align*}
          \io (n_1(\cdot, t) - n_2(\cdot, t))^2
      \le \io (n_{0,1} - n_{0,2})^2
          + 2(C_1 + C_2) \int_0^t \io (n_1(\cdot, t) - n_2(\cdot, t))^2,
    \end{align*}
  so that the statement follows by Grönwall's inequality for $C \defs C_1 + C_2$.
\end{proof}
By virtue of Lemmata \ref{lem:positivity} and 
\ref{lem:gronwall}, we can assure uniqueness of solutions to \eqref{our_system}. 
\begin{lemma}\label{lm:unique}
  Assume \eqref{cond:Om}, $\gamma \gt 0$, $\beta \in (0, 1)$, $C^{1 + \beta}(\partial \Omega) \ni g \gt 0$
  as well as $C^0(\Ombar) \ni n_0 \gt 0$ and let $T \in (0, \infty]$.
  Then there is at most one solution $(n, c)$ to~\eqref{our_system} which is of regularity \eqref{reg_nc_t}.
\end{lemma}
\begin{proof}
We fix $T_0\in(0,T)$. Suppose $(n_1, c_1)$ and $(n_2, c_2)$ are two such solutions.
  The (weak) maximum principle asserts $n_1, n_2 \ge 0$ while moreover by Lemma~\ref{lm:mass_conservation},
  $\io n_1(\cdot, t) = \io n_0 = \io n_2(\cdot, t)$ for all $t \in [0, T_0)$, hence $n_1 \not\equiv 0 \not\equiv n_2$.
  Then Lemma~\ref{lem:existence-c(t)} gives $0 \lt c_1, c_2 \lt \gamma$,
  so that we may apply Lemma~\ref{lem:positivity} to obtain positivity of $n_1$ and $n_2$.
  Thus, Lemma~\ref{lem:gronwall} and Lemma~\ref{lem:existence-c(t)} assert $n_1 = n_2$ and $c_1 = c_2$ in $\Ombar \times [0, T_0)$,
  respectively.
\end{proof}
We are now in the position to show existence of classical solutions to \eqref{our_system} for less regular initial data. 
\begin{lemma}\label{lem:existence_generalinitialdata}
Assume \eqref{cond:Om}, $\gamma \gt 0$ , $\beta \in (0, 1)$ and $C^{1 + \beta}(\partial \Omega) \ni g \gt 0$
and let $n_0\in C^0(\Ombar)$ be positive. 
Then there is a unique global solution $(n, c)$ of regularity \eqref{reg_nc} to~\eqref{our_system}.
\end{lemma}
\begin{proof}
 First, fix $T \gt 0$.
 We choose a sequence $\{n_{0,j}\}_{j\inℕ}$ of functions $n_{0,j}\in C^{1+β}(\Ombar)$
 satisfying $∂_{ν}n_{0,j} = n_{0,j}(γ-c_{0,j})g$ on $\partial\Om$ (where $c_{0,j}$ solves \eqref{ceq} with $n=n_{0,j}$),
 admitting a constant $C_1>0$ such that 
 \begin{equation}\label{eqn:n0j}
  \f1{C_1} \le n_{0,j} \le C_1 \qquad\text{in } \Om \text{ for all } j\in ℕ
 \end{equation}
and fulfilling $n_{0,j}\to n_0$ in $\Lom2$. By $c_j$ and $n_j$ we denote the corresponding solutions to \eqref{ceq} and \eqref{neq} in $\Om\times(0,T)$, respectively (whose existence is given by Lemma \ref{lem:locex_c1betadata}). 
Lemma \ref{lem:boundedness-n-in-Lp} and Lemma \ref{lem:W2p-estimate-c} show that with some $C_2>0$ 
\begin{equation}\label{eqn:c2}
 \norm[\Lom{∞}]{n_j(\cdot,t)} \le C_2, \quad \norm[W^{1,∞}(\Om)]{c_j(\cdot,t)}\le C_2\qquad \text{for all } j\in ℕ \text{ and } t\in(0,T).
\end{equation}
Moreover, from \eqref{ceq} and \eqref{eqn:c2} we conclude that 
\begin{equation}\label{eqn:c2squared}
 \norm[L^{∞}(\Om\times(0,T))]{Δc_j} = \norm[L^{∞}(\Om\times(0,T))]{n_jc_j} \le C_2^2\qquad \text{for all } j\in ℕ.
\end{equation}
At the same time, \eqref{eqn:c2} also implies that with some $C_3>0$, 
\begin{equation}\label{eqn:C3}
 \norm[L^2(\Om\times(0,T))]{n_j}\le C_3\qquad \text{for all } j\in ℕ.
\end{equation}
Given any $ε>0$, from Lemma \ref{lem:hoelderbound:n}(ii) we obtain $C_4>0$ such that 
\begin{equation}\label{eqn:C4}
 \normm{C^{β',\f{β'}2}(\Ombar\times[ε,T])}{n_j}\le C_4 \qquad \text{for all } j\inℕ
\end{equation}
and Lemma \ref{lem:ex:c-u} ensures the existence of $C_5>0$ such that 
\begin{equation}\label{eqn:C5}
 \normm{C^{2+β',\f{β'}2}(\Ombar\times(2ε,T])}{c_j}\le C_5\qquad \text{for all } j\inℕ.
\end{equation}
Again for $ε>0$ and based on \eqref{eqn:c2} and \eqref{eqn:C5}, 
Lemma \ref{lem:n-c1beta} provides us with $C_6>0$ such that 
\begin{equation}\label{eqn:C6}
 \normm{C^{1+β',\f{1+β'}2}(\Ombar\times[4ε,T])}{n_j}\le C_6 \qquad \text{for all } j\inℕ.
\end{equation}
For any compact set $K\subset \Om\times(0,∞)$ we can introduce $ζ_K\in C_c^{∞}(\Om\times(0,T))$ such that $ζ_K|_K\equiv 1$ and note that $ζ_Kn$ solves 
\[
 \begin{cases}
  (ζn)_t - Δ(ζn) = f\defs -ζ∇n\cdot∇c -ζnΔc - 2∇ζ\cdot∇n +ζ_t n & \text{in } \Om \times (0, T),\\
  (ζn)(\cdot,0)=0 & \text{in } \Om,\\
  ∂_{ν}(ζn)=0 &\text{on } ∂\Om\times(0,T).
 \end{cases}
\]
 Relying on \eqref{eqn:C6} (with $ε$ so small that $\supp ζ_K\cap (\Ombar\times[0,4ε]) = \emptyset$), \eqref{eqn:C5}, \eqref{eqn:c2squared} to estimate $\normm{C^{β,\f{β}2}(\Ombar\times[0,T])}{f}$, we can use \cite[Theorem~IV.5.3]{LSU} to get $C_7>0$ satisfying 
\begin{equation}\label{eqn:C7}
 \norm[C^{2+β,1+\f{β}2}(K)]{n_j}\le \normm{C^{2+β,1+\f{β}2}(\Ombar\times[0,T])}{ζn_j}\le C_7\qquad \text{for all  } j\in ℕ.
\end{equation}
Finally, using \eqref{eqn:n0j}, \eqref{eqn:c2}, \eqref{eqn:c2squared}
and Lemma \ref{lem:gronwall} (in conjunction with Lemma~\ref{lem:positivity}), we find $C_8>0$ such that 
\begin{equation}\label{eqn:C8}
 \norm[\Lom2]{n_{j_1}(\cdot,t)-n_{j_2}(\cdot,t)}\le \ure^{C_8t} \norm[\Lom2]{n_{0,j_1}-n_{0,j_2}} \qquad \text{for all } t\in[0,T],\;\, j_1,\,j_2\inℕ.
\end{equation}
Applying compactness arguments (supplemented by \eqref{eqn:C7}, \eqref{eqn:C6}) in combination with \eqref{eqn:C8} (for $n_j$), and \eqref{eqn:C5} (for $c_j$), we see that 
\begin{equation}\label{eqn:conv}
 (n_j,c_j)\to (n,c) \qquad \text{as } j\to \infty
\end{equation}
in $\left(C^{2,1}(\Om\times(0,T))\cap C^{1,0}(\Ombar\times(0,T))\cap C^0([0,T);\Lom2) \right)\times C^{2,0}(\Ombar\times(0,T))$. 
Additionally, by virtue of \eqref{eqn:C3}, \cite[Theorem~2]{benedetto_Local} assures us that additionally $n\in C^0(\Ombar\times[0,T))$.
The convergence in \eqref{eqn:conv} guarantees that $(n,c)$ solves \eqref{our_system}.

Now letting $T \ra \infty$, thanks to Lemma~\ref{lm:unique}, we can extend $(n, c)$ to a global-in-time solution.
The same lemma also asserts uniqueness of global solutions.
\end{proof}

\subsection{Global boundedness. Proof of Theorem~\ref{th:global_ex}}\label{sec:globbd-thm-gex}
We finally show global boundedness of solutions to \eqref{our_system}, and then obtain Theorem \ref{th:global_ex}. 
\begin{lemma}\label{lm:c2_bdd}
  Assume \eqref{cond:Om}, $\gamma \gt 0$ , $\beta \in (0, 1)$, $C^{1 + \beta}(\partial \Omega) \ni g \gt 0$
  as well as $\con0 \ni n_0 \gt 0$.
  Then the solution $(n, c)$ to~\eqref{our_system} constructed in Lemma~\ref{lem:existence_generalinitialdata}
  is bounded in the sense that there exist $C \gt 0$ and $\beta' \in (0, 1)$ such that
   \begin{align*}
    \|n(\cdot,t)\|_{L^\infty(\Om)}+\|c(\cdot,t)\|_{W^{1,\infty}(\Om)}\le C \qquad \text{for all }t>0\\
 \text{and}\quad   \|n(\cdot, t)\|_{\con{\beta'}} + \|c(\cdot, t)\|_{\con{2+\beta'}} \le C
    \qquad \text{for all $t \ge 1$.}
  \end{align*}
\end{lemma}
\begin{proof}
  Lemma~\ref{lem:boundedness-n-in-Linfty} states
  \begin{align*}
    \sup_{t \in (0, \infty)} \|n(\cdot, t)\|_{\leb \infty} \lt \infty,
  \end{align*}
  so that Lemma~\ref{lem:W2p-estimate-c} moreover asserts
  \begin{align*}
    \sup_{t \in (0, \infty)} \|c(\cdot, t)\|_{\sob{2}{N+1}} \lt \infty,
  \end{align*}
 which already shows the first boundedness statement.

  Let $t_0 \gt 0$. Since $n(\cdot, \cdot + t_0)$ solves \eqref{neq} with initial datum $n_0(\cdot + t_0)$
  and as $\sob{2}{N+1} \embed \con1$,
  we furthermore have
  \begin{align*}
    \|n\|_{C^{\beta_1, \frac{\beta_1}{2}}(\Ombar \times [t_0+\frac12, t_0+1])} \le C_1
  \end{align*}
  for some $C_1 \gt 0$ and $\beta_1 \in (0, 1)$ (both not depending on $t_0$) by Lemma~\ref{lem:hoelderbound:n}(ii).
  Combined with Lemma~\ref{lem:bounds-c(t)} this implies the statement.
\end{proof}

Theorem~\ref{th:global_ex} is now merely a consequence of the lemmata above:
\begin{proof}[Proof of Theorem~\ref{th:global_ex}]
Existence and uniqueness follow from Lemma~\ref{lem:existence_generalinitialdata}, 
positivity is given by Lemmata~\ref{lem:existence-c(t)} and \ref{lem:positivity}
and the boundedness statement is entailed by Lemma~\ref{lm:c2_bdd}.
\end{proof}

\section{Large time behaviour. Proof of Theorem~\ref{th:conv}}\label{sec:largetime}
Our goal, proving Theorem~\ref{th:conv}, will be a consequence of the following
\begin{lemma} \label{lem:ddt_l2_conv}
  Assume \eqref{cond:Om}.
  There is $C \gt 0$ such that
  for all $\gamma \gt 0$, $\beta \in (0, 1)$, $C^{1+\beta}(\partial \Omega) \ni g \gt 0$, $m \gt 0$
  and $\con0 \ni n_0 \gt 0$ with $\io n_0 = m$ the following holds:

  Denote the global solution of \eqref{our_system} given by Theorem~\ref{th:global_ex} by $(n, c)$
  and the solution to the stationary problem~\eqref{prob:stationary} with $\io n_\infty = m$,
  constructed in~\cite{braukhoff_lankeit}, by $(n_\infty, c_\infty)$.
  Then
  \begin{align}\label{eq:ddt_l2_conv:est_n}
        K(m, g, \gamma) \int_0^\infty \io (n - n_\infty)^2
    \le \io (n_0 - n_\infty)^2 
  \end{align}
  and
  \begin{align}\label{eq:ddt_l2_conv:est_c}
          \io |\nabla (c - c_\infty)|^2 + \frac{m}{2 \ure^\gamma} \io (c - c_\infty)^2 
    &\le  \frac{m \gamma^2 \ure^{2\gamma}}{2|\Omega| \|n_\infty\|_{\leb\infty}^2}  \io (n - n_\infty)^2
  \end{align}
  holds throughout $(0, \infty)$ with
  \begin{align*}
    K(m, g, \gamma)
    \defs \frac{\lambda_1}2
          - \gamma \left(
              C \|g\|_{\leb \infty} \max\left\{\gamma^2 \|g\|_{\leb \infty}^2, 1 \right\}
              + \frac{m \gamma \ure^{2\gamma}}{2|\Omega|}
            \right),
  \end{align*}
  where $\lambda_1$ denotes the smallest positive eigenvalue of $-\Delta$ with homogeneous Neumann boundary conditions.
\end{lemma}
\begin{proof}
  Any such solution $(n, c)$ fulfills
  \begin{align*}
      0 = \Delta (c - c_\infty) - (n - n_\infty) c - n_\infty (c - c_\infty) \qquad \text{in $\Omega \times (0, \infty)$},
  \end{align*}
  hence
  \begin{align}\label{eq:ddt_l2_conv:nabla_c_2}
          \io |\nabla (c - c_\infty)|^2
    &=    - \int_{\partial \Omega} (c - c_\infty)^2 g
          - \io c (n - n_\infty) (c - c_\infty)
          - \io n_\infty (c - c_\infty)^2 \notag \\
    &\le  - \frac12 \io n_\infty (c - c_\infty)^2 
          + \frac{\gamma^2}{2} \io \frac{(n - n_\infty)^2}{n_\infty}\qquad\text{in } (0,∞)
  \end{align}
  by Young's inequality and as $c \le \gamma$ in $\Omega \times (0, \infty)$.

  By \cite[Theorem~1.1 and Lemma~4.1]{braukhoff_lankeit} there is $\alpha_\infty \gt 0$
  such that $n_\infty = \alpha_\infty \ure^{c_\infty} \ge \alpha_\infty$.
  Since $\io n_\infty = m$ and $1 \le \ure^{c_\infty} \le \ure^\gamma$,
  we can estimate $\frac{m}{\ure^\gamma |\Omega|} \le \alpha_\infty \le \frac{m}{|\Omega|}$
  and hence $\|n_\infty\|_{\leb \infty}^2 \le \alpha_\infty \frac{m}{|\Omega|} \ure^{2\gamma}$.
  Thus, \eqref{eq:ddt_l2_conv:nabla_c_2} already implies~\eqref{eq:ddt_l2_conv:est_c}.

  Moreover, by Lemma~\ref{lm:weak_testing}
  (applied to $n - n_\infty$ with $f \defs -n_j \nabla c - n_\infty \nabla (c - c_\infty)$, $\psi \equiv 0$ and $p \defs 2$),
  \begin{align}\label{eq:ddt_l2_conv:nabla_n_2}
    &\pe  \frac12 \io (n(\cdot, t) - n_\infty)^2 - \frac12 \io (n_0 - n_\infty)^2 + \io |\nabla (n - n_\infty)|^2 \notag \\
    &=    \int_0^t \io [(n - n_\infty) \nabla c
          + n_\infty \nabla (c - c_\infty)] \cdot \nabla (n - n_\infty) \notag \\
    &=    \frac12 \int_0^t \io \nabla (n - n_\infty)^2 \cdot \nabla c
          + \int_0^t \io n_\infty \nabla (n - n_\infty) \cdot \nabla (c - c_\infty) \notag\\
    &\sfed I_1(t) + I_2(t) \qquad\text{for } t \in (0,∞).
  \end{align}
  By Lemma~\ref{lm:trace_thm}(ii)
  (with $q \defs 2$, $\lambda \defs \frac13 \lt \frac12$ and $\eps \defs \frac{1}{4\gamma \|g\|_{\leb \infty}}$),
  we have 
  \begin{align}\label{eq:ddt_l2_conv:i1}
          I_1(t)
    &=    - \frac12 \int_0^t \io (n - n_\infty)^2 \Delta c
          + \int_0^t \int_{\partial \Omega} (n - n_\infty)^2 (\gamma - c) g \notag \\
    &\le  - \frac12 \int_0^t \io (n - n_\infty)^2 nc
          + \gamma \|g\|_{\leb \infty} \int_0^t \int_{\partial \Omega} (n - n_\infty)^2 \notag \\
    &\le  \frac14 \int_0^t \io |\nabla (n - n_\infty)|^2
          + C \gamma \|g\|_{\leb \infty} \left(1 + \gamma^2 \|g\|_{\leb \infty}^2 \right) \int_0^t \io (n - n_\infty)^2
          \qquad \text{for } t \in (0,∞) 
  \end{align}
  for some $C \gt 0$ only depending on $\Omega$.
  Furthermore,
  \begin{align}\label{eq:ddt_l2_conv:i2}
          I_2(t)
    &\le  \frac14 \int_0^t \io |\nabla (n - n_\infty)|^2
          + \|n_\infty\|_{\leb \infty}^2 \int_0^t \io |\nabla (c - c_\infty)|^2 \notag \\
    &\le  \frac14 \int_0^t \io |\nabla (n - n_\infty)|^2
          + \frac{m \gamma^2 \ure^{2\gamma}}{2|\Omega|} \int_0^t \io (n - n_\infty)^2
          \qquad \text{for } t \in (0,∞) 
  \end{align}
  by \eqref{eq:ddt_l2_conv:est_c}.
  Thus, \eqref{eq:ddt_l2_conv:nabla_n_2} combined with~\eqref{eq:ddt_l2_conv:i1} and~\eqref{eq:ddt_l2_conv:i2}
  asserts that
  \begin{align*}
        \frac12 \io (n(t) - n_\infty)^2
    \le \frac12 \io (n_0 - n_\infty)^2
        - \frac12 \int_0^t \io \nabla (n - n_\infty)^2
        + \left( \frac{\lambda_1}{2} - K(m, g, \gamma) \right) \int_0^t \io (n - n_\infty)^2
  \end{align*}
  holds for all $t \in (0, \infty)$.
  Since $\io \varphi^2 \le \frac1{\lambda_1} \io |\nabla \varphi|^2$ for all $\varphi \in \sob12$ with $\io \varphi = 0$ by Poincar\'e's inequality,
  this implies~\eqref{eq:ddt_l2_conv:est_n} due to the fact that $\io (n - n_\infty) = m - m = 0$.
\end{proof}

Thanks to the energy estimate obtained in Lemma~\ref{lem:ddt_l2_conv} and Hölder bounds prepared earlier,  
we can achieve convergence of solutions of \eqref{our_system}. 
\begin{lemma}\label{lem:conv}
  Assume \eqref{cond:Om}. 
  Let $C$ and $K$ be as in Lemma~\ref{lem:ddt_l2_conv}.
  If $\gamma \gt 0$, $\beta \in (0, 1)$, $C^{1+\beta}(\partial\Om)\ni g \gt 0$, $m \gt 0$ and $\con0 \ni n_0 \gt 0$
  are such that $\io n_0 = m$ and $K(m, g, \gamma) \gt 0$,
  then
  \begin{align*}
    (n(\cdot, t), c(\cdot, t)) \ra (n_\infty, c_\infty) \qquad \text{in $\con0 \times \con2$ as $t \ra \infty$},
  \end{align*}
  where $(n, c)$ and $(n_\infty, c_\infty)$
  again denote the solutions to \eqref{our_system} and \eqref{prob:stationary} with $\io n_\infty = m$, respectively.
\end{lemma}
\begin{proof}
 Since Lemma~\ref{lm:c2_bdd} asserts uniform continuity of $(1, \infty) \ni t \mapsto \io (n(\cdot, t) - n_\infty)^2$,
  we infer from \eqref{eq:ddt_l2_conv:est_n} and the assumption that $K > 0$, that 
  \begin{align*}
    n(\cdot, t) &\ra n_\infty\hp{c} \qquad \text{in $\leb2$ as $t \ra \infty$}.
  \intertext{Together with \eqref{eq:ddt_l2_conv:est_c}, this in turn implies}
    c(\cdot, t) &\ra c_\infty\hp{n} \qquad \text{in $\sob12$ as $t \ra \infty$}.
  \end{align*}
  The statement follows from Lemma~\ref{lm:c2_bdd} in conjunction with 
  the compact embeddings $\con{\beta'} \embed \con0$ and  $\con{2+\beta'} \embed \con2$ for $\beta' \in (0, 1)$.
\end{proof}

\begin{proof}[Proof of Theorem~\ref{th:conv}]
  Let $C$ and $K$ be as in Lemma~\ref{lem:ddt_l2_conv}. 
  Given $m, C_g \gt 0$ we can find $\gamma_0 \gt 0$
  and given $\gamma_0 \gt 0$ we can find $m, C_g \gt 0$
  such that
  if $m, g, \gamma$ comply with \eqref{eq:conv:cond_n0} and \eqref{eq:conv:cond_g_gamma},
  then $K(m, g, \gamma) > 0$. 
  Uniform convergence of $n$, $c$ and $\nabla c$ is then a consequence of Lemma~\ref{lem:conv}.
\end{proof}

\appendix
\section{An interpolation inequality}\label{sec:interpolation}
Owing to the nonhomogenous boundary conditions in~\eqref{ceq},
in the proofs above (namely in the proofs of Lemmata~\ref{lem:boundedness-n-in-Lp}, \ref{lem:gronwall} and \ref{lem:ddt_l2_conv})
we needed to handle certain boundary integrals.
The following interpolation lemma allows us to estimate such terms
and may turn out to be useful for the analysis of other systems as well,
including but not limited to those with nonhomogenous boundary conditions.
In fact, similar arguments have been employed for a chemotaxis system with homogeneous boundary conditions
in the proof of \cite[Proposition~3.2]{IshidaEtAlBoundednessQuasilinearKeller2014}.

\begin{lemma} \label{lm:trace_thm}
  Suppose \eqref{cond:Om}, $q \in (0, 2]$ and
  \begin{align} \label{eq:trace_thm:cond_lambda}
    \lambda \in \left(0, \frac{q}{2N + 2q - Nq} \right).
  \end{align}

  \begin{enumerate}
    \item[(i)] 
      There exists $C \gt 0$ such that
      \begin{align*}
            \|\varphi\|_{\leb[\partial \Omega]2}
        \le C \|\nabla \varphi\|_{\leb 2}^{1-\lambda} \|\varphi\|_{\leb q}^\lambda + C \|\varphi\|_{\leb q}
        \qquad \text{for all $\varphi \in \sob12$}.
      \end{align*}
     
    \item[(ii)] 
      We may find $C' \gt 0$ such that for any $\eps \gt 0$
      \begin{align*}
            \|\varphi\|_{\leb[\partial \Omega]2}
        \le \eps \|\nabla \varphi\|_{\leb 2} + C' \min\{1, \eps\}^{-\frac{1-\lambda}{\lambda}} \|\varphi\|_{\leb q}
        \qquad \text{for all $\varphi \in \sob12$}.
      \end{align*}
  \end{enumerate}
  Here and below, we denote $\left( \io \varphi^r \right)^\frac1r$ by $\|\varphi\|_{\leb r}$ for all $r \gt 0$.
\end{lemma}
\begin{proof}
  Ad (i):
    By \eqref{eq:trace_thm:cond_lambda} we have
    \begin{align*}
      \theta \defs \frac12 - \left( \frac{q}{2N + 2q - Nq} - \lambda \right) \cdot \frac{2N + 2q - Nq}{2q} \in \left(0, \frac12\right),
    \end{align*}
    so that (I.9.1) and Theorem~I.9.4 in \cite{LionsMagenesNonhomogeneousBoundaryValue1972} yield
    \begin{align*}
        \left[\sob12, \leb2\right]_\theta
      = \sob{1-\theta}{2}
      \embed \leb[\partial \Omega]2.
    \end{align*}
    That is, by Proposition~I.2.3 in \cite{LionsMagenesNonhomogeneousBoundaryValue1972} we may find $c_1 \gt 0$ with
    \begin{align}\label{embeddingineq}
          \|\varphi\|_{\leb[\partial \Omega]2}
      \le c_1 \|\varphi\|_{\sob12}^{1-\theta} \|\varphi\|_{\leb 2}^\theta
      \qquad \text{for all $\varphi \in \sob12$}.
    \end{align}
    Since $q \in (0, 2]$,
    Lemma~2.2 and Lemma~2.3 in \cite{LiLankeitBoundednessChemotaxisHaptotaxis2016} allow us to further estimate
    \begin{align}\label{gnineq1}
          \|\varphi\|_{\leb2}
      \le c_2 \|\varphi\|_{\sob 12}^{1-a} \|\varphi\|_{\leb q}^a
      \qquad \text{for all $\varphi \in \sob12$}
    \end{align}
    for some $c_2 \gt 0$, where
    \begin{align*}
        a \defs 1 - \frac{\frac1q - \frac12}{\frac1q + \frac1N - \frac12}
      = \frac{\frac1N}{\frac1q + \frac1N - \frac12} 
      = \frac{2q}{2N + 2q - Nq}
      \in (0, 1).
    \end{align*}

    By combining the estimates \eqref{embeddingineq} and \eqref{gnineq1}, we get
    \begin{align*}
          \|\varphi\|_{\leb[\partial \Omega]2}
      \le c_1 c_2 \|\varphi\|_{\sob 12}^{1 - \theta a} \|\varphi\|_{\leb q}^{\theta a}
      \qquad \text{for all $\varphi \in \sob12$},
    \end{align*}
    where we again make use of \cite[Lemma~2.2]{LiLankeitBoundednessChemotaxisHaptotaxis2016} to obtain
    \begin{align*}
          \|\varphi\|_{\leb[\partial \Omega]2}
      \le c_3 \|\nabla \varphi\|_{\leb 2}^{1 - \theta a} \|\varphi\|_{\leb q}^{\theta a} + c_3 \|\varphi\|_{\leb q}
      \qquad \text{for all $\varphi \in \sob12$}
    \end{align*}
    for some $c_3 \gt 0$.

    Finally, the statement follows upon noting that
    \begin{align*}
          \theta a
      &=  \left[ \frac12 - \left( \frac{q}{2N + 2q - Nq} - \lambda \right) \cdot \frac{2N + 2q - Nq}{2q} \right]
          \cdot \frac{2q}{2N + 2q - Nq} \\
      &=  \frac{q}{2N + 2q - Nq} - \left( \frac{q}{2N + 2q - Nq} - \lambda \right)
      =   \lambda.
    \end{align*}

  Ad (ii):
    Young's inequality (with exponents $\frac{1}{1-\lambda}, \frac1\lambda$) asserts that
    \begin{align*}
          a b
      =   ((1-\lambda) \tilde \eps)^{1-\lambda} a \cdot \left((1-\lambda) \tilde \eps \right)^{-(1-\lambda)} b
      \le \tilde \eps a^\frac{1}{1-\lambda} + \lambda \left((1-\lambda) \tilde \eps \right)^{-\frac{1-\lambda}{\lambda}} b^\frac1\lambda
    \end{align*}
    holds for all $a, b \ge 0$ and all $\tilde \eps \gt 0$.

    Thus, using part~(i) and 
    choosing $a \defs \|\nabla \varphi\|_{\leb 2}, b \defs \|\varphi\|_{\leb q}$, $\tilde \eps \defs \frac{\eps}{C}$
    and $C' \gt 0$ large enough
    proves part (ii) and hence the lemma.
\end{proof}

\begin{remark}
  For the case $q=2$, which we have used in the critical Lemma~\ref{lem:ddt_l2_conv},
  the condition \eqref{eq:trace_thm:cond_lambda} is optimal, see \cite[Theorem~9.51]{LionsMagenesNonhomogeneousBoundaryValue1972}.
\end{remark}

\section*{Acknowledgements}
The first author is partially supported by the German Academic Scholarship Foundation
and by the Deutsche Forschungsgemeinschaft within the project \emph{Emergence of structures and advantages in
cross-diffusion systems}, project number 411007140.

{\footnotesize 
\def\cprime{$'$}

}
\end{document}